\documentclass[a4paper,11pt]{article}
\usepackage{amsmath}
\usepackage{amsfonts,amsthm,amssymb,eucal}
\usepackage{indentfirst}
\usepackage[unicode]{hyperref}

\usepackage{mathtext}

\usepackage[cp1251]{inputenc}
\usepackage[T2A]{fontenc}

\newtheorem{theorem}{Theorem}

\newtheorem{lemma}{Lemma}
\newtheorem{definition}{Definition}

\numberwithin{equation}{section}

\newcommand{\Z}{\mathbb{Z}}
\newcommand{\Q}{\mathbb{Q}}
\newcommand{\N}{\mathbb{N}}
\newcommand{\R}{\mathbb{R}}
\newcommand{\A}{\mathbb{A}}
\newcommand{\IA}{\tilde{\mathbb{A}}}
\newcommand{\Comp}{\mathbb{C}}
\newcommand{\Pl}{\mathcal{P}}
\newcommand{\IPl}{\tilde{\mathcal{P}}}
\newcommand{\M}{\tilde{M}}

\title{On distribution of points with conjugate algebraic integer coordinates close to planar curves}
\author{V.\,Bernik, F.\,G\"otze, A.\,Gusakova}
\date{}

\begin{document}

\maketitle

\begin{abstract}

Let $\varphi:\R\rightarrow \R$ be a continuously differentiable function on an interval $J\subset\R$ and let $\boldsymbol{\alpha}=(\alpha_1,\alpha_2)$ be a point with algebraic conjugate integer coordinates of degree $\leq n$ and of height $\leq Q$. Denote by $\M^n_\varphi(Q,\gamma, J)$ the set of points $\boldsymbol{\alpha}$ such that $|\varphi(\alpha_1)-\alpha_2|\leq c_1 Q^{-\gamma}$. In this paper we show that for a real $0<\gamma<1$ and any sufficiently large $Q$ there exist positive values $c_2<c_3$, which are independent of $Q$, such that $c_2\cdot Q^{n-\gamma}<\# \M^n_\varphi(Q,\gamma, J)< c_3\cdot Q^{n-\gamma}$.

\end{abstract}

\section{Introduction}

An important and interesting topic in the theory of Diophantine approximation is the distribution of algebraic numbers and algebraic integers \cite{BeBeGo10, Cas57, Spr67, Sch80}. In this paper we consider problems related to the distribution of points with algebraic conjugate integer coordinates in the plane.

Let us start with some useful notation. Let $n$ be a positive integer and $Q>1$ be a sufficiently large real number. Given a polynomial $P(t)=a_nt^n+\ldots+a_1t+a_0\in\Z[t]$ denote by $H(P)=\max\limits_{0\leq j \leq n}{|a_j|}$ the height of the polynomial $P$, and by $\deg P$ the degree of the polynomial $P$. We define the following classes of integer polynomials with bounded height and degree:
\[
\Pl_{n}(Q):=\{P\in\Z[t]:\,\deg P\leq n,\, H(P)\leq Q\}.
\]
\[
\IPl_{n}(Q):=\{P\in\Z[t]:\,\deg P=n,\, H(P)\leq Q,\, a_n=1\}.
\]
Denote by $\#S$ the cardinality of a finite set $S$ and by $\mu_k S$ the Lebesgue measure of a measurable set $S\subset \R^k$, $k\in\N$. Furthermore, denote by $c_j>0$, $j\in \N$ positive constants independent of $Q$. We are going to use the Vinogradov symbol $A\ll B$, which means that there exists a constant $c>0$ independent of $A$ and $B$ such that $A\leq c\cdot B$. We will write $A\asymp B$ when $A\ll B$ and $B\ll A$.

Now let us introduce the concept of an algebraic integer point. A point $\boldsymbol{\alpha}=(\alpha_1,\alpha_2)$ is called an {\it algebraic point} if $\alpha_1$ and $\alpha_2$ are roots of the same irreducible polynomial $P\in\Z[t]$. If the leading coefficient $a_n$ of polynomial $P$ is equal to 1, then a point $\boldsymbol{\alpha}$ is called an {\it algebraic integer point}. The polynomial $P$ is called the minimal polynomial of the point $\boldsymbol{\alpha}$. Denote by $\deg(\boldsymbol{\alpha})=\deg P$ the degree of the point $\boldsymbol{\alpha}$ and by $H(\boldsymbol{\alpha})=H(P)$ the height of the point $\boldsymbol{\alpha}$. We denote by $\A^2$ (respectively $\IA^2$) the set of algebraic points (respectively integer algebraic points). Furthermore, we define the following sets:
\[
\A_n^2(Q):=\left\{\boldsymbol{\alpha}\in \A^2:\, \deg{\boldsymbol{\alpha}}\leq n,\, H(\boldsymbol{\alpha})\leq Q\right\},
\] 
\[
\IA_n^2(Q):=\left\{\boldsymbol{\alpha}\in \IA^2:\, \deg{\boldsymbol{\alpha}}= n,\, H(\boldsymbol{\alpha})\leq Q\right\}.
\] 

The problem of determining the number of integer points in regions and bodies of $\R^k$ can be naturally generalized to estimating the number of rational points in domains of Euclidean spaces. Let $f:J_0\rightarrow \R$ be a continuously differentiable function defined on a finite open interval $J_0\subset\R$. Let us consider the following set:
\[
N_f(Q,\gamma, J):=\left\{\left(\frac{p_1}{q},\,\frac{p_2}{q}\right)\in\Q^2:\,0<q\leq Q,\, \frac{p_1}{q}\in J,\,\left|f\left(\frac{p_1}{q}\right)-\frac{p_2}{q}\right|<Q^{-\gamma}\right\},
\]
where $J\subset J_0$ and $0\leq\gamma < 2$. Thus, the quantity $\# N_f(Q,\gamma, J)$ denotes the number of rational points with bounded denominators lying within a certain neighborhood of the curve parametrized by $f$. The problem is to estimate the value $\# N_f(Q,\gamma, J)$. This question was considered by Huxley  \cite{H96}, Vaughan, Velani \cite{VoVe06} and Beresnevich, Dickinson, Velani \cite{BeDiVe07} and with some additional restrictions on the function $f$ it was proved that
\[
\# N_f(Q,\gamma,J)\asymp Q^{3-\gamma}.
\]
This result was obtained using methods of metric number theory introduced by Schmidt in \cite{Sch80}. 

The following natural extension of this question is the problem of distribution of algebraic points $\boldsymbol{\alpha}\in\A_n^2(Q)$ near smooth curves. Let $\varphi:J_0\rightarrow \R$ be a continuously differentiable function defined on a finite open interval $J_0\subset\R$ satisfying the conditions:
\begin{equation}\label{eq1_1}
\sup\limits_{x\in J_0}|\varphi'(x)|:= c_{4} < \infty,\qquad\#\{x\in J_0:\varphi(x)=x\} := c_{5}<\infty.
\end{equation}
Define the following set:
\[
M^n_\varphi(Q,\gamma, J):= \left\{\boldsymbol{\alpha}\in\A_n^2(Q):\,\alpha_1\in J, |\varphi(\alpha_1)-\alpha_2|\ll Q^{-\gamma}\right\},
\]
where $J\subset J_0$. The goal is to estimate the number $\# M^n_\varphi(Q,\gamma, J)$. A first attempt to solve this problem for $0<\gamma\leq\frac12$ has been made in \cite{BeGoKu14}. This result was complemented by lower bound of the right order for $0<\gamma<\frac34$ \cite{BeGoGu16a} and finally by lower and upper bounds of the same order for $0<\gamma<1$ \cite{BeGoGu16b}. We are going to state the final result in the following form: for any smooth function $\varphi$ with conditions \eqref{eq1_1} we have $\# M^n_\varphi(Q,\gamma, J) \asymp Q^{n+1-\gamma}$ for $Q>Q_0(n, J, \varphi,\gamma)$ and $0< \gamma < 1$.

We will consider the same problem in the case of integer algebraic points. Let
\[
\M^n_\varphi(Q,\gamma, J):= \left\{\boldsymbol{\alpha}\in\IA_n^2(Q):\,\alpha_1\in J, |\varphi(\alpha_1)-\alpha_2|<c_1Q^{-\gamma}\right\}.
\]

\begin{theorem}\label{main}
For any smooth function $\varphi$ with conditions \eqref{eq1_1} there exist positive values $c_{2},c_{3}>0$ such that
\[
c_{2}\cdot Q^{n-\gamma}<\# \M^n_\varphi(Q,\gamma, J) < c_{3}\cdot Q^{n-\gamma}
\]
for $Q>Q_0(n, J,\varphi,\gamma)$, $0< \gamma < 1$, $n\ge 3$ and a sufficiently large constant $c_1>0$ in definition of the set $\M^n_\varphi(Q,\gamma, J)$.
\end{theorem}

It should be noted that a lower bound of $\# \M^n_\varphi(Q,\gamma, J)$ for $0< \gamma \leq \frac12$ was obtained earlier \cite{GoGu16}. Hence we will assume for the lower bound that $\gamma >\frac12$.

The proof of Theorem \ref{main} is based on the following idea. We consider the strip $L^n_\varphi(Q,\gamma, J):= \left\{\mathbf{x}\in\R^2: x_1\in J, |\varphi(x_1)-x_2|<c_1Q^{-\gamma}\right\}$ and cover it with squares $\Pi=I_1\times I_2$  with sides of length $\mu_1 I_1=\mu_1 I_2 = c_6Q^{-\gamma}$, where $c_6=c_1 / \left(\frac12+c_4\right)$ (the full description of this scheme is given in \cite{BeGoGu16b}). Thus, in order to prove Theorem \ref{main} we need to estimate the number of integer algebraic points lying in such a square $\Pi$.

Let us consider here a more general case, namely, the case of a rectangle $\Pi=I_1\times I_2$, where $\mu_1 I_i = c_6Q^{-\gamma_i}$. 

\begin{theorem}\label{th1}
Let $\Pi=I_1\times I_2$ be a rectangle with a midpoint $\mathbf{d}$ and sides $\mu_1 I_i=c_6Q^{-\gamma_{i}}$, $i=1,2$. Then for $0<\gamma_1,\gamma_2< 1$, $n\ge 2$ and $Q>Q_0(n,\gamma, \mathbf{d})$ the estimate
\[
\# \left(\IA_n^2(Q)\cap\Pi\right) < c_{7}\cdot Q^{n}\mu_2\Pi
\]
holds, where
\[
c_{7}=2^{3n+8}n^2\rho_n(d_1)\rho_n(d_2)\left|d_1-d_2\right|^{-1}\quad\text{and}\quad\rho_n(x)=\left(\left(|x|+1\right)^{n+1}-1\right)\cdot |x|^{-1}.
\]
\end{theorem}

The case of lower bound is more difficult. It is easy to prove that there exist rectangles $\Pi$ of size $\mu_2 \Pi \asymp Q^{-1}$ such that $\# \left(\A_n^2(Q)\cap\Pi)\right)= 0$, and since $\IA_n^2(Q)\subset\A_n^2(Q)$ we have $\# \left(\IA_n^2(Q)\cap\Pi)\right)= 0$ for such rectangles. It means that we cannot obtain non-zero lower bounds for all rectangles $\Pi$. In particular, it is easy to show that certain neighborhoods of algebraic points of small height and small degree do not contain any other algebraic points $\boldsymbol{\alpha}\in\A_n^2(Q)$ at all. In order to avoid such domains we use the concept of a {\it $(v_1,v_2)$-special} square, which has been introduced in \cite{BeGoGu16b}.

\begin{definition}
Let $\Pi=I_1\times I_2$ be a square with midpoint $\mathbf{d}$, $d_1\neq d_2$ and sides $\mu_1 I_1 = \mu_1 I_2 = c_6Q^{-\gamma}$ such that $\frac12< \gamma< 1$. We say that the square $\Pi$ satisfies the {\it $(l,v_1,v_2)$-condition} if $v_1+v_2=1$ and there exist at most $\delta^3\cdot 2^{l+3}Q^{1+2\lambda_{l+1}}\mu_2\Pi$ polynomials $P\in\Pl_2(Q)$ of the form $P(t)=a_2t^2+a_1t+a_0$ satisfying the inequalities  
\[
\begin{cases}
|P(x_{0,i})|<h\cdot Q^{-v_i},\quad i=1,2,\\
\delta Q^{\lambda_{l+1}}\leq|a_2|<\delta Q^{\lambda_l}
\end{cases}
\]
for some point $\mathbf{x}_0\in \Pi$, where $\delta = 2^{-L-17}h^{-2}\cdot (d_1-d_2)^2$, $L=\left[\textstyle\frac{3-2\gamma}{1-\gamma}\right]$ and
\[
\lambda_l=
\begin{cases}
1-\frac{(l-1)(1-\gamma)}{2},\quad 1\leq l\leq L+1,\\
\gamma-\frac12,\quad l= L+2,\\
0,\quad l\ge L+3.
\end{cases}
\]
\end{definition}

\begin{definition}
The square $\Pi=I_1\times I_2$ with sides $\mu_1 I_1 = \mu_1 I_2 = c_6Q^{-\gamma}$ such that $\frac12< \gamma< 1$ is called a {\it $(v_1,v_2)$-special} square if it satisfies the {\it $(l,v_1,v_2)$-condition} for all $1\leq l\leq L+2$.
\end{definition}

The following theorem can be proved for  {\it $(v_1,v_2)$-special} squares.

\begin{theorem}\label{th2}
For all {\it $\left(\textstyle\frac12,\textstyle\frac12\right)$-special} squares $\Pi=I_1\times I_2$ with midpoints $\mathbf{d}$, $d_1\neq d_2$ and sides $\mu_1 I_1 = \mu_1 I_2 =c_6Q^{-\gamma}$, where $\frac12< \gamma< 1$ and $c_6>c_0(n,\mathbf{d})$, there exists a value $c_{8}=c_{8}(n,\mathbf{d},\gamma)>0$ such that
\[
\#\left(\IA_n^2(Q)\cap\Pi\right)> c_{8}\cdot Q^{n}\mu_2\Pi
\]
for $Q>Q_0(n,\mathbf{d},\gamma)$ and $n\ge 3$.
\end{theorem}

\section{Auxiliary statements}

This section contains several lemmas which will be used to prove Theorems \ref{th1} and \ref{th2}. Some of them are related to geometry of numbers, see \cite{Cas97}. The first paper discussing approximations by algebraic integers are due to Davenport and Schmidt \cite{DavSch}. Recently their approach has been further developed by Bugeaud \cite{Bug03} and we shall use ideas of this paper. 

\begin{lemma}[Minkowski 2nd theorem on successive minima]\label{lm_Minkowski}
Let $K\subset\R^n$ be a bounded central symmetric convex body with  successive minima $\tau_1,\ldots,\tau_n$ and volume $V(K)$.
Then
\[
\frac{2^n}{n!}\leq\tau_1\tau_2\ldots\tau_nV(K)\leq 2^n.
\]
\end{lemma}

For a proof, see  \cite[pp. 203]{Cas97}, \cite[pp. 59]{Gruber}.

\begin{lemma}[Bertrand postulate]\label{lm_Bertran}
For any $n\in\N$, $n \ge 2$ there exists a prime $p$ such that $n<p<2n$.
\end{lemma}

It was proved by Chebyshev in 1850. A proof can be found, for example, in \cite[Theorem 2.4]{Nesterenko}.

\begin{lemma}[Eisenstein's criterion]\label{lm_Eisenstein}
Let $P\in\Z[t]$ be a polynomial of the form $P(t)=a_nt^n+a_{n-1}t^{n-1}+\ldots +a_1t+a_0$. If there exists a prime number $p$ such that:
\begin{equation}\label{eq2_1}
\begin{cases}
a_n\not\equiv 0 \mod{p},\\
a_i\equiv 0 \mod{p},\quad i=0,\ldots, n-1\\
a_0\not\equiv 0 \mod{p^2},
\end{cases}
\end{equation}
then $P$ is irreducible over the rational numbers.
\end{lemma}

For a proof see \cite{Eis}, \cite[Theorem 2.1.3]{Prasolov}.

\begin{lemma}\label{lm_polynomial}
Consider a point $x\in\R$ and a polynomial $P$ with zeros $\alpha_1,\alpha_2,\ldots,\alpha_n$ 
where $|x-\alpha_1| = \min\limits_{i} |x-\alpha_i|$. Then
\[
|x-\alpha_1| \le n\cdot|P(x)|\cdot|P'(x)|^{-1}.
\]
\end{lemma}

\begin{proof}
Evaluate the polynomial $P$ and its derivative $P'$ at the point $x\neq\alpha_i$ for $i=1,2,\ldots,n$. Since
\[
|P(x)|\cdot|P'(x)|^{-1}\leq\sum\limits_{i=1}^n|x-\alpha_i|^{-1}\leq n\cdot|x-\alpha_1|^{-1},
\]
we obtain
\[
|x-\alpha_1| \le n\cdot|P(x)|\cdot|P'(x)|^{-1}.
\] 
\end{proof}

\begin{lemma}[see \cite{Fel51}]\label{lm5}
For any subset of roots $\alpha_{i_1},\ldots,\alpha_{i_s}$, $1\leq s\leq n$, of the polynomial $P(t)=a_nt^n+\ldots+a_1t+a_0$ we have $\prod\limits_{j=1}^{s}|\alpha_{i_j}|\leq (n+1)2^n\cdot H(P)\cdot |a_n|^{-1}$.  
\end{lemma}

\begin{lemma}[see {\cite{BeGoGu16b}}]\label{lm_BeGoGu}
Let $\Pi=I_1\times I_2$ be a square with midpoint $\mathbf{d}$, $d_1\neq d_2$ and sides $\mu_1 I_1 = \mu_1 I_2 =c_6Q^{-\gamma}$, where $\frac12< \gamma< 1$ and $c_6>c_0(n,\mathbf{d})$. Given positive values $v_1,v_2$ such that $v_1+v_2=n-1$, let $L=L_n(Q,\delta_n,\mathbf{v},\Pi)$ be the set of points $\mathbf{x}\in\Pi$ such that there exists a polynomial  $P\in\Pl_{n}(Q)$ satisfying the following system of inequalities:
\[
\begin{cases}
|P(x_i)|< h_n\cdot Q^{-v_i},\\
\min\limits_i\{|P'(x_i)|\}<\delta_n\cdot Q,\quad i=1,2,\\
\end{cases}
\]
where $h_n = \sqrt{\frac32(|d_1|+|d_2|)\cdot \max\left(1,3|d_1|,3|d_2|\right)^{n^2}}$. If $\Pi$ is a {\it $\left(\frac{v_1}{n-1},\frac{v_2}{n-1}\right)$-special} square, then
\[
\mu_2 L<\textstyle\frac14 \cdot\mu_2\Pi
\]
for $\delta_n<\delta_0(n,\mathbf{d})$ and $Q>Q_0(n,\mathbf{v},\mathbf{d},\gamma)$.
\end{lemma}

\begin{lemma}[see {\cite{BeGoGu16b}}]\label{lm6}
Let $G=G(\mathbf{d},\mathbf{K})$, where $|d_1-d_2|>\varepsilon_1>0$, be a set of points $\mathbf{b}=(b_1,b_0)\in\Z^2$ such that
\[
|b_1d_i+b_0|\leq K_i,\quad i=1,2.
\]
Then
\[
\# G\leq \left(4\varepsilon_1^{-1}K_1+1\right)\cdot\left(4K_2+1\right).
\]
\end{lemma}

\section{Proof of Theorem \ref{th1}}

Assume that $\#\left(\IA_n^2(Q)\cap\Pi\right) \ge c_{7}\cdot Q^{n}\mu_2\Pi$. Take an integer algebraic point $\boldsymbol{\alpha}\in\IA_n^2(Q)\cap\Pi$ with minimal polynomial $P$. Let us give an estimate for the polynomial $P$ at the points $d_1$ and $d_2$. Since $\alpha_i\in I_i$, we have
\[
|P^{(k)}(\alpha_{i})|\leq\sum\limits_{j=k}^{n-1}{\textstyle\frac{j!}{(j-k)!}\cdot|a_{j}|\cdot |\alpha_i|^{j-k}}+\textstyle\frac{n!}{(n-k)!}\cdot |\alpha_i|^{n-k}< \textstyle\frac{n!}{(n-k)!}\cdot \rho_n(d_i)\cdot Q
\]
for all $1\leq k\leq n$ and $Q>Q_0$. From these estimates and Taylor expansion of $P$ in the intervals $I_i$, $i=1,2$, we obtain the following inequality:
\begin{multline}\label{eq3_1}
|P(d_i)|\leq\sum\limits_{k=1}^n{\left|\textstyle\frac{1}{k!}P^{(k)}(\alpha_i)(d_i-\alpha_i)^k\right|}<\\
< \sum\limits_{k=1}^n{2^{-k}\textstyle{k\choose n}\rho_n(d_i)\cdot Q\mu_1 I_i}\leq 2^n\rho_n(d_i)\cdot Q\mu_1 I_i.
\end{multline}

Let us fix the vector $\mathbf{A}_{1}=(1,a_{n-1},\ldots,a_{2})$, where $a_{n-1},\ldots,a_{2}$ are the coefficients of the polynomial $P\in \IPl_{n}(Q)$. Denote by $\IPl_{n}(Q,\mathbf{A}_{1})\subset \IPl_n(Q)$ the subclass of polynomials $P$ with the same vector of coefficients $\mathbf{A}_{1}$ such that $P$ satisfies \eqref{eq3_1}. The number of subclasses $\IPl_{n}(Q,\mathbf{A}_{1})$ is equal to the number of vectors $\mathbf{A}_{1}$, which for $Q>Q_0$ can be estimated as follows:
\begin{equation}\label{eq3_2}
\#\{\mathbf{A}_{1}\}=(2Q+1)^{n-2}< 2^{n-1}\cdot Q^{n-2}.
\end{equation}
It should also be noted that every point of the set $\IA_n^2(Q)\cap\Pi$ corresponds to some polynomial $P\in\IPl_n(Q)$ that satisfies \eqref{eq3_1}. On the other hand, every polynomial $P\in\IPl_n(Q)$ satisfying (\ref{eq3_1}) corresponds to at most $n^2$ points of the set $\IA_n^2(Q)\cap\Pi$. This allows us to write
\[
c_{7}\cdot Q^{n+1}\mu_2\Pi< \# \left(\IA_n^2(Q)\cap\Pi\right)\leq n^2\sum\limits_{\mathbf{A}_{1}}\# \IPl_n(Q,\mathbf{A}_{1}).
\]
Thus, by the estimate \eqref{eq3_2} and Dirichlet's box principle applied to vectors $\mathbf{A}_{1}$ and polynomials $P$ satisfying \eqref{eq3_1}, there exists a vector $\mathbf{A}_{1,0}$ such that
\begin{equation}\label{eq3_2}
\#\IPl_{n}(Q,\mathbf{A}_{1,0})\ge c_{7}\cdot 2^{-n+1}n^{-2}Q^{2}\mu_2\Pi.
\end{equation}
Let us find an upper bound for the value $\#\IPl_n(Q,\mathbf{A}_{1,0})$. To do this, we fix some polynomial $P_0\in\IPl_{n}(Q,\mathbf{A}_{1,0})$ and consider the difference between the polynomials $P_0$ and $P_j\in\IPl_{n}(Q,\mathbf{A}_{1,0})$ at the points $d_i$, $i=1,2$. From the estimate \eqref{eq3_1} it follows that
\[
|P_0(d_i)-P_j(d_i)|=|(a_{0,1}-a_{j,1})d_i+(a_{0,0}-a_{j,0})|\leq 2^{n+1}\rho_n(d_i)\cdot Q\mu_1 I_i.
\]
Thus the number of different polynomials $P_j\in\IPl_{n}(Q,\mathbf{A}_{1,0})$ does not exceed the number of integer solutions of the following system:
\[
|b_1d_i+b_0|\leq 2^{n+1}\rho_n(d_i)\cdot Q\mu_1 I_i,\quad i=1,2.
\]
Now, let us use Lemma \ref{lm6} for $K_i=2^{n+1}\rho_n(d_i)\cdot Q\mu_1 I_i$. Since $\mu_1 I_i = c_6 Q^{-\gamma_i}$ and $\gamma_i < 1$, we have $K_i\ge 2^{n+1}\rho_n(d_i)c_6\cdot Q^{1-\gamma_i}>\max\{\varepsilon_1,1\}$ for $Q>Q_0$. This implies 
\[
j \leq 2^{2n+8}|d_1-d_2|^{-1}\rho_n(d_1)\rho_n(d_2)\cdot Q^2\mu_2\Pi.
\]
It follows that $\#\IPl_{n}(Q,\mathbf{A}_{1,0})\leq 2^{2n+8}|d_1-d_2|^{-1}\rho_n(d_1)\rho_n(d_2)\cdot Q^2\mu_2\Pi$, which contradicts inequality \eqref{eq3_2} for $c_{7}=2^{3n+8}n^2\rho_n(d_1)\rho_n(d_2)|d_1-d_2|^{-1}$. This leads to
\[
\#\left(\IA_n^2(Q)\cap\Pi\right)< c_{7}\cdot Q^{n}\mu_2\Pi.
\]

\section{Proof of Theorem \ref{th2}}

Since $d_1\neq d_2$ we can assume that for $Q>Q_0$ the following inequality
\[
|x_1-x_2|>\varepsilon=\textstyle\frac{|d_1-d_2|}{2}
\]
is satisfied for every point $\mathbf{x}\in\Pi$.

In order to prove the Theorem \ref{th2} we use Lemma \ref{lm_BeGoGu}. Given positive constants $u_1$ and $u_2$ satisfying the condition $u_1+u_2=n-2$ let $L=L_{n-1}(Q,\delta,\mathbf{u},\Pi)$ be the set of points $\mathbf{x}\in\Pi$ such that the following system of inequalities
\begin{equation}\label{eq4_1}
\begin{cases}
|P(x_i)|< h_{n-1}\cdot Q^{-u_i},\\
\min\limits_i\{|P'(x_i)|\}<\delta\cdot Q,\quad i=1,2,\\
\end{cases}
\end{equation}
has a solution in polynomials $P\in\Pl_{n-1}(Q)$. Lemma \ref{lm_BeGoGu} implies that the measure of the set $L$ can be estimated as
\[
\mu_2 L\leq\textstyle\frac14\cdot \mu_2\Pi
\]
for $\delta<\delta_0(n-1,\mathbf{d})<1$ and $Q>Q_0(n-1,\mathbf{u},\mathbf{d},\gamma)$.

Let us consider the set $B=\Pi\setminus L$. Using Minkowski's linear form theorem \cite[Ch. 2, \S 3]{Sch80} for every point $\mathbf{x}\in \Pi$ there exists a polynomial $P\in \Pl_{n-1}(Q)$ such that 
\[
|P(x_i)|\leq h_{n-1}\cdot Q^{-u_i},\quad i=1,2.
\]
Thus, we can assert that for every point $\mathbf{x} \in B$ there exists an irreducible polynomial $P\in\Pl_{n-1}(Q)$ such that
\[
\begin{cases}
|P(x_{i})|<h_{n-1}\cdot Q^{-u_i},\\
|P'(x_{i})|>\delta\cdot Q,\quad i=1,2,
\end{cases}
\]
and $\mu_2\, B \ge\textstyle\frac34\cdot \mu_2\Pi$.

Consider an arbitrary point $\mathbf{x} \in B$ and let us examine the successive minima
$\tau_1,\ldots,\tau_n$ of the compact convex set defined by
\begin{equation}\label{eq4_2}
\begin{cases}
|a_{n-1}x_{i}^{n-1}+\ldots+a_1x_{i}+a_0|\leq h_{n-1}Q^{-u_i},\\
|(n-1)a_{n-1}x_{i}^{n-2}+\ldots+2a_2x_{i}+a_1| \leq Q,\quad i=1,2,\\
|a_{n-1}|,\ldots,|a_2|\leq Q.
\end{cases}
\end{equation}
Assume that $\tau_1\leq\delta$. Then for sufficiently small $\delta$ there exists a polynomial $P\in\Pl_{n-1}(Q)$ such that the inequalities
\[
\begin{cases}
|P(x_{i})| < \delta h_{n-1} Q^{-u_i}<h_{n-1}Q^{-u_i},\\
|P'(x_{i})|<\delta Q,\quad i=1,2,\\
H(P) < Q
\end{cases}
\]
hold. This leads to a contradiction, since $\mathbf{x}\not\in L$. Thus $\tau_1>\delta$. Since the volume of the compact convex set defined by the inequalities \eqref{eq4_2} is at least $2^n$, it follows from Lemma \ref{lm_Minkowski} that $\tau_1\ldots\tau_{n}\leq 1$ and $\tau_{n}\leq \delta^{-n+1}$. Thus, by definition of successive minima, we can choose $n$ linearly independent polynomials $P_j(t)=a_{j,n-1}t^{n-1}+\ldots+a_{j,1}t+a_{j,0}\in\Pl_{n-1}(Q)$, $1\leq j\leq n$, satisfying 
\begin{equation}\label{eq4_3}
\begin{cases}
|P_j(x_{i})|\leq \delta^{-n+1}h_{n-1}Q^{-u_i},\\
|P_j'(x_{i})| \leq \delta^{-n+1}Q,\quad i=1,2,\\
|a_{j,k}|\leq \delta^{-n+1}Q, \quad 4\leq k\leq n-1. 
\end{cases}
\end{equation}
 Using  well-known estimates from the geometry of numbers, see \cite[pp. 219]{Cas97}, we obtain for the polynomials $P_j$, $1\leq j\leq n$ the inequality:
\[
\Delta = \det|(a_{j,k-1})^n_{j,k=1}|\leq n!.
\]
For a prime $p$ not dividing $\Delta$, Lemma \ref{lm_Bertran} yields
\begin{equation}\label{eq4_4}
n! <p< 2n!.
\end{equation}

Consider the system of linear equations for the $n$ variables $\theta_1,\ldots,\theta_n$
\begin{equation}\label{eq4_5}
\begin{cases}
x_{i}^n+p\sum\limits_{j=1}^{n}{\theta_j P_j(x_{i})}=p(n+1)\delta^{-n+1}h_{n-1}Q^{-u_i},\\
nx_{i}^{n-1}+p\sum\limits_{j=1}^{n}{\theta_j P_j'(x_i)}=pQ+ p\sum\limits_{j=1}^{n}{|P_j'(x_{i})|},\quad i=1,2,\\
\sum\limits_{j=1}^{n}{\theta_j a_{j,k-1}}=0,\quad 5\leq k\leq n.
\end{cases}
\end{equation}
It should be mentioned that in case $n=3$ the values $|P_j'(\alpha_{j,1})|$ and $|P_j'(\alpha_{j,2})|$ are equal, where $\alpha_{j,1}$ and $\alpha_{j,2}$ are the roots of the polynomial $P_j$. It means that one of the equations numbered 2 and 3 can be removed.

In order to fined the determinant of this system, we transform it as follows.
Multiply the equation numbered as $k=5, 6, \ldots, n$ by $p\cdot x_1^{k-1}$ (respectively by $p \cdot x_2^{k-1}$) and subtract it from the first (respectively the second) equation of the system \eqref{eq4_5}. Similarly multiply the equation numbered as $k=5, 6, \ldots, n$ by $p\cdot (k-1)x_1^{k-2}$ (respectively by $p \cdot (k-1)x_2^{k-2}$) and subtract it from the third (respectively the fourth) equation. After these transformations the determinant of the system \eqref{eq4_5} can be written as
\[
\hat{\Delta}(\mathbf{x})=p^4\cdot 
\begin{vmatrix}
\sum\limits_{k=0}^3a_{1,k}x_1^k & \dots & \sum\limits_{k=0}^3a_{n,k}x_1^k \\
\sum\limits_{k=0}^3a_{1,k}x_2^k & \dots & \sum\limits_{k=0}^3a_{n,k}x_2^k \\
\sum\limits_{k=1}^3k\cdot a_{1,k}x_1^{k-1} & \dots & \sum\limits_{k=1}^3k\cdot a_{n,k}x_1^{k-1} \\
\sum\limits_{k=1}^3k\cdot a_{1,k}x_2^{k-1} & \dots & \sum\limits_{k=1}^3k\cdot a_{n,k}x_2^{k-1} \\
a_{1,4} & \dots & a_{n,4}\\
\vdots & \ddots & \vdots \\
a_{1,n-1} & \dots & a_{n,n-1}
\end{vmatrix}.
\]
Let us transform the first four rows of this matrix as follows. Multiply the third (respectively the fourth) row by $\frac13 x_1$ (respectively by $\frac13 x_2$) and subtract it from the first (respectively the second) row. Then we subtract the first (respectively the third) row from the second (respectively the fourth) row and obtain the following determinant:
\[
\hat{\Delta}(\mathbf{x})=\textstyle\frac{p^4(x_2-x_1)^2}{9}\cdot\begin{vmatrix}
a_{1,2}x_1^2+2a_{1,1}x_1+3a_{1,0} & \dots & a_{n,2}x_1^2+2a_{n,1}x_1+3a_{n,0} \\
a_{1,2}(x_2+x_1)+2a_{1,1} & \dots & a_{n,2}(x_2+x_1)+2a_{n,1} \\
3a_{1,3}x_1^2+2a_{1,2}x_1+a_{1,1} & \dots & 3a_{n,3}x_1^2+2a_{n,2}x_1+a_{n,1} \\
3a_{1,3}(x_2+x_1)+2a_{1,2}& \dots & 3a_{n,3}(x_2+x_1)+2a_{n,2} \\
a_{1,4} & \dots & a_{n,4}\\
\vdots & \ddots & \vdots \\
a_{1,n-1} & \dots & a_{n,n-1}
\end{vmatrix}.
\]
Now let us subtract the second row multiplied by $x_1$ from the first row. Similarly, subtract the fourth row multiplied by $\frac12$ from the third row. Then subtract the third row multiplied by $\frac{x_2+x_1}{x_1^2}$ from the fourth row, and finally subtract the fourth row multiplied by $x_1x_2$, $x_2+x_1$ and $\frac32x_1-\frac12x_2$ from the first, the second and the third row respectively. We obtain the equation
\begin{equation}\label{eq4_6}
\hat{\Delta}(\mathbf{x})=p^4(x_2-x_1)^4\cdot 
\begin{vmatrix}
a_{1,0} & \dots & a_{n,0}\\
\vdots & \ddots & \vdots \\
a_{1,n-1} & \dots & a_{n,n-1}
\end{vmatrix}
=p^4(x_2-x_1)^4\Delta >0,
\end{equation}
since the polynomials $P_j$, $1\leq j\leq n$ are linearly independent and $|x_1-x_2|>\varepsilon>0$. By \eqref{eq4_6} the system \eqref{eq4_5} has a unique solution
$(\theta_1,\ldots,\theta_n)$. 

Consider $n$ integers $s_1,\ldots,s_n$ satisfying
\begin{equation}\label{eq4_7}
|\theta_j - s_j|\leq 1,\quad 1\leq j \leq n.
\end{equation}
and construct the following polynomial with integer coefficients
\[
P(t)=t^n+p\cdot\sum\limits_{j=1}^n{s_jP_j(t)}=t^n+p\cdot(a_{n-1}t^{n-1}+\ldots+a_1t+a_0),
\]
where $a_k=\sum\limits_{j=1}^{n}{s_ja_{j,k}}$, $0\leq k\leq n-1$. 

The polynomial $P$ is irreducible if it satisfies the conditions of Lemma \ref{lm_Eisenstein}. Let us show that there exists a suitable combinations of the coefficients $s_j$. Clearly, the first and the second condition of \eqref{eq2_1} hold for any $s_j$. It remains to show that
$a_0=s_1a_{1,0}+\ldots+s_na_{n,0}$ is not divisible by $p$. Since $p$ doesn't divide $\Delta$, there exists a number $1\leq j\leq n$ such that $a_{j,0}$ is not divisible by $p$. From the condition \eqref{eq4_7}, we have two possible values for $s_j$, which can be denoted as $s_j^1$, $s_j^2=s_j^1+1$. Since $a_{j,0}$ is not divisible by $p$, either $a_0^1=s_1a_{1,0}+\ldots+a_{j,0}s_j^1+\ldots+a_{n,0}s_n$
or $a_0^2=s_1a_{1,0}+\ldots+a_{j,0}s_j^2+\ldots+a_{n,0}s_n$ is also not divisible by $p$. Therefore, choosing $s_j$ in this manner yields an irreducible polynomial $P$.

We finally derive bounds for $|P(x_i)|$, $|P'(x_i)|$ and $H(P)$. By the inequalities \eqref{eq4_3}, \eqref{eq4_5} and \eqref{eq4_7} we obtain the following estimates:
\begin{equation}\label{eq4_8}
p\delta^{-n+1}h_{n-1}Q^{-u_i}\leq|P(x_i)|\leq p(2n+1)\delta^{-n+1}h_{n-1}Q^{-u_i},\quad i=1,2,
\end{equation} 
\begin{equation}\label{eq4_9}
pQ\leq |P'(x_i)|\leq (p+2pn\delta^{-n+1})Q,\quad i=1,2.
\end{equation}

We now estimate the height $H(P)$.
By equation $5$ to  $n$ of the system \eqref{eq4_5}, inequalities \eqref{eq4_3} and \eqref{eq4_7} we have:
\begin{equation}\label{eq4_10}
|a_k|\leq n\delta^{-n+1}Q,\quad 4\leq k\leq n-1.
\end{equation}
It remains to estimate $|a_0|$, $|a_1|$, $|a_2|$ and $|a_3|$. By \eqref{eq4_8}-- \eqref{eq4_10} and inequalities $|x_i|\leq |d_i|+\frac12$ we get: 
\begin{multline}\label{eq4_11}
|a_3x_i^3+a_2x_i^2+a_1x_i+a_0|\leq |P(x_i)|+\sum\limits_{k=4}^{n}{\left(|d_i|+1\right)^k\cdot |a_k|}<c_{9,i}Q,\\
|3a_3x_i^2+2a_2x_i+a_1|\leq |P'(x_i)|+\sum\limits_{k=4}^{n}{k\left(|d_i|+1\right)^k\cdot |a_k|}<\\
<c_{10,i}Q,\quad i=1,2,
\end{multline}
where 
\[
c_{9,i}=\begin{cases}
h_{n-1},\quad n = 3,\\
2n\delta^{-n+1}h_{n-1}\left(|d_i|+1\right)^{n},\quad n > 3;
\end{cases}
c_{10,i}=\begin{cases}
p+2pn\delta^{-n+1}h_{n-1},\quad n = 3,\\
4pn^2\delta^{-n+1}h_{n-1}\left(|d_i|+1\right)^{n},\quad n > 3.
\end{cases}
\]

Consider the system of linear equations for $a_0$, $a_1$, $a_2$ and $a_3$
\begin{equation}\label{eq4_12}
\begin{cases}
a_3x_i^3+a_2x_i^2+a_1x_i+a_0=l_{1,i},\\
3a_3x_i^2+2a_2x_i+a_1=l_{2,i},\quad i=1,2.\\
\end{cases}
\end{equation}

Since the determinant of the system \eqref{eq4_12} does not vanish, there exists a unique solution. We solve the system \eqref{eq4_12} subject to the estimates \eqref{eq4_11} and inequalities $|x_i|\leq |d_i|+\frac12$. We obtain
\[
|a_k|<c_{11}Q,\quad 0\leq k\leq 3.
\]
Hence, by \eqref{eq4_4} and \eqref{eq4_10}, we find that
\begin{equation}\label{eq4_13}
H(P)<\max\{c_{11}, n\delta^{-n+1}\}Q=Q_1.
\end{equation}

Consider the roots $\alpha_{1},\ldots,\alpha_{n}$ of the polynomial $P$, where $|x_i-\alpha_{i}|=\min\limits_{j}{|x_i-\alpha_{j}|}$, $i=1,2$.
By Lemma \ref{lm_polynomial}, the following estimates hold
\[
|x_i-\alpha_{i}|\leq n|P(x_i)|\cdot|P'(x_i)|^{-1},\quad i=1,2.
\]
By \eqref{eq4_8} and \eqref{eq4_9}, we have
\begin{equation}\label{eq4_14}
|x_i-\alpha_{i}|< n(2n+1)\delta^{-n+1}h_{n-1}Q^{-u_i-1}<c_{12}Q^{-u_i-1},\quad i=1,2.
\end{equation}
where $c_{12}=n(2n+1)\delta^{-n+1}h_{n-1}$. Let us prove that $\alpha_1,\alpha_2 \in \R$ for $u_1=u_2=\frac{n-2}{2}$. Assume the converse: let $\alpha_i\in\Comp$, then the number $\overline{\alpha_i}$ being complex conjugate to $\alpha_i$ is also a root of the polynomial $P$. Hence, by \eqref{eq4_13}, \eqref{eq4_14} and Lemma \ref{lm5} we conclude that
\[
|P(x_i)|=\prod\limits_{j=1}^n|x_i-\alpha_j|\leq c_{12}^2Q^{-n}\cdot c_{13}\cdot Q=c_{13}c_{12}^2\cdot Q^{-n+1}.
\]
This inequality contradicts \eqref{eq4_8} for $Q>Q_0$. 

Let us choose a maximal system of algebraic integer points $\Gamma = \{\boldsymbol{\gamma}_1,\ldots,\boldsymbol{\gamma}_t\}\subset\IA_n^2(Q_1)$ satisfying the condition that rectangles $\sigma(\boldsymbol{\gamma}_k)=\{|x_i-\gamma_{k,i}|<c_{12}Q^{-\frac{n}{2}},i=1,2\}$, $1\leq k\leq t$ do not intersect. 
Furthermore, let us introduce the expanded rectangles
\[
\sigma'(\boldsymbol{\gamma}_k)=\left\{|x_i-\gamma_{k,i}|<2c_{12}Q^{-\frac{n}{2}},i=1,2\right\},\quad k = 1,\ldots, t,
\]
and show that
\begin{equation}\label{eq4_16}
B\subset\bigcup_{k=1}^t \sigma'(\boldsymbol{\gamma}_k).
\end{equation}
To prove this fact, we are going to show that for any point $\mathbf{x}_1 \in B_1$ there exists a point $\boldsymbol{\gamma}_k\in\Gamma$ such that $\mathbf{x}_1\in\sigma'(\boldsymbol{\gamma}_k)$. 
Since $\mathbf{x}_1 \in B_1$, there is an algebraic integer point $\boldsymbol{\alpha}\in\IA_n^2(Q_1)$ satisfying the inequalities \eqref{eq4_14}. Thus, either $\boldsymbol{\alpha}\in\Gamma$ and $\mathbf{x}_1\in\sigma'(\boldsymbol{\alpha})$, or there exists a point $\boldsymbol{\gamma}_k\in\Gamma$ satisfying
\[
|\alpha_i-\gamma_{k,i}|\leq c_{12}Q^{-\frac{n}{2}},\quad i=1,2, 
\]
which implies that $\mathbf{x}_1\in\sigma'(\boldsymbol{\gamma}_k)$. Hence, from \eqref{eq4_16} and the estimate $\mu_2\,B\ge\frac34\cdot\mu_2\,\Pi$ we have
\[
\textstyle\frac34\cdot \mu_2\Pi \leq\mu_2 B\leq \sum\limits_{k=1}^t{\mu_2\sigma_1(\boldsymbol{\gamma}_k)}\leq t\cdot 2^4c_{12}^2Q^{-n},
\]
which yields the estimate
\[
\#\left(\IA_n^2(Q_1)\cap\Pi\right)\ge t \ge c_{8}\cdot Q^{n}\mu_2\Pi.
\]

\newpage
Anna Gusakova,\\
Institute of Mathematics,\\ 
Belorussian Academy of Sciences,\\
Surganova str. 11, 220072 Minsk, Belarus\\
E-mail: gusakova.anna.0@gmail.com\\
\\
Vasili Bernik,\\
Institute of Mathematics,\\ 
Belorussian Academy of Sciences,\\
Surganova str. 11, 220072 Minsk, Belarus\\
E-mail: vasili.bernik@mail.com\\
\\
Friedrich G{\"o}tze,\\
Faculty of Mathematics, University of Bielefeld, \\
PO Box 100131, 33501 Bielefeld, Germany\\
E-mail: goetze@math.uni-bielefeld.de
\end{document}